\documentclass[a4paper,12pt]{article}

\usepackage[english]{babel}
\newcommand*{\arXiv}[1]{\href{http://arxiv.org/abs/#1}{arXiv:#1}}
\usepackage[margin=2cm]{geometry}
\usepackage{type1cm}	
\usepackage{titlesec}   
\usepackage{fancyhdr}   

\usepackage[dvipsnames]{xcolor}

\usepackage{graphicx}   
\usepackage{titling}    
\usepackage{tabularx}   
\usepackage[shortlabels]{enumitem} 

\usepackage{amsmath, amsthm, amssymb} 
\usepackage{mathtools}  

\usepackage{tikz-cd}    
\usepackage{pgfplots}
\pgfplotsset{compat=1.18}
\usepackage[breakable]{tcolorbox}   
\usepackage{standalone} 
\usetikzlibrary{decorations.pathmorphing}
\usetikzlibrary{calc, arrows, matrix}
\usetikzlibrary{external}
\usepackage[unicode=true, pdfborder={0 0 0}, bookmarksdepth=-1]{hyperref} 
\usepackage[capitalise,nameinlink]{cleveref}

\newtheoremstyle{mystyle}
  {6pt}{15pt}
  {\it}
  {}
  {\bf}
  {.}
  {1em}
  {}

\theoremstyle{mystyle}	
\newtheorem{theorem}{Theorem}[section]

\newtheorem{corollary}[theorem]{Corollary}

\newtheorem{lemma}[theorem]{Lemma}

\newtheorem{question}{Question}

\newtheorem{conjecture}{Conjecture}
\theoremstyle{definition}
\newtheorem{definition}[theorem]{Definition}

\DeclareMathOperator{\Freq}{Freq}
\begin{document}

\newcommand{\F}{\mathcal{F}} 
\newcommand{\fq}{\Freq_{\F}}
 
\title{Further analysis on the second frequency of union-closed set families}
\author{Saintan Wu\thanks{
National Taiwan University, r12201025@ntu.edu.tw}}
\date{\today}
\maketitle

\begin{abstract}
    The Union-Closed Sets Conjecture, also known as Frankl's conjecture, asks whether, for any union-closed set family $\mathcal{F}$ with $m$ sets, there is an element that lies in at least $\frac{1}{2}\cdot m$ sets in $\mathcal{F}$. In 2022, Nagel posed a stronger conjecture that within any union-closed family whose ground set size is at least $k$, there are always $k$ elements in the ground set that appear in at least $\frac{1}{2^{k-1}+1}$ proportion of the sets in the family.

    Das and Wu showed that this conjecture is true for $k\geq 3$ and $k=2$ if $|\mathcal{F}|$ is outside a particular range. In this companion paper, we analyse further when $\mathcal{F}$ fails Nagel's conjecture for $k=2$ via linear programming.
\end{abstract}

\section{Introduction}
A \textit{union-closed} set family $\mathcal{F}$ is a set family satisfying $A,B\in \mathcal{F}\implies A\cup B\in \mathcal{F}$. Nagel \cite{nagel2023notes} conjectured that:

\begin{conjecture}\label{conj:Nagel}
    For any union-closed set family $|\mathcal{F}|$ with ground set $\geq k$, there are $k$ elements, each of them lies in at least $\frac{1}{2^{k-1}+1}|\mathcal{F}|$ many sets in $\mathcal{F}$.
\end{conjecture}

Let $F_k(\mathcal{F})$ denote the ratio of sets in $\mathcal{F}$ that contain the $k$-th most frequent element. So the above conjecture states that $F_k(\mathcal{F})\geq \frac{1}{2^{k-1}+1}$ for such families.

Notice that when $k=1$, it conjectures the well-known Union-Closed Sets Conjecture, that every nonempty union-closed family has an element appearing in at least $\frac{|\mathcal{F}|}{2}$ sets.

\cite{our_main_paper} gives a notion about $k$-good set, which here we will only use $2$-good sets:
\begin{definition}
    Given a union-closed family $\mathcal{F}$, a set $S$ is \textbf{$2$-good} for $\mathcal{F}$ if it does not contain $1$, and for each set $A$ that is not $\varnothing$ nor $\{1\}$, $S\cap A \neq\varnothing$.

    A $2$-good set is called \textbf{minimal} if its proper subsets are not $2$-good.
\end{definition}

Then, in the paper, \cite{our_main_paper} proved, via the entropic method for large families and shattering method for small families, that

\begin{theorem}
\label{thm:main}
    Let $k \geq 2$, and let $\mathcal{F}$ be a union-closed set family with $|\mathcal{F}| = m$ and $|\cup_{F \in \mathcal{F}} F| \ge k$. Then the $k$th-most frequent element lies in at least $\frac{m}{2^{k-1} + 1}$ sets in $\mathcal{F}$, with equality only if $\mathcal{F}$ is a near-$k$-cube, provided:
    \begin{itemize}
        \item[(i)] $k \ge 3$, or
        \item[(ii)] $k = 2$, and either $m \le 44$ or $m \ge 114$.        
    \end{itemize}
\end{theorem}

We still did not know if there is a counterexample for $k=2$ when $|\mathcal{F}|\in [45,113]$. We will analyse this case more carefully and narrow the range.

Our main result in this paper is the following:
\begin{theorem}\label{thm:sidemain}
    Let $\F$ be union-closed with $|\cup_{F \in \F} F| \ge 2$, and that $\F$ is not a near-$2$-cube. Then, if $f_2(\F) \le \frac13$, we must have 
    \begin{enumerate}[(i)]
        \item $81 \le |\F| \le 113$, and
        \item All minimal $2$-good sets $S$ have size $4$. 
    \end{enumerate}
\end{theorem}

In the following sections, we will focus on the case when $\F$ has $f_2(\F)\leq \frac13$ and is in the case that \cref{thm:main} didn't cover. That is, when $|\F|\in [45,113]$ and any of its minimal $2$-good sets $S$ has size $4$ or $5$.

We will further assume that $\varnothing\in \mathcal{F}$, and $1$ is the most abundant element in sets of $\mathcal{F}$.

\subsection*{Paper structure and notations}
In Section 2, we will briefly explain the idea behind the proof. Then, we will prove several lemmas mentioned in the heuristics. After that, two separate cases $|S|=4$ and $|S|=5$ will be done in \cref{sec:s=4} and \cref{sec:s=5}, respectively.

For simplicity, for any set $S$ and element $x$, we use $S+x$ for $S\cup\{x\}$, and $S-x$ for $S\setminus\{x\}$. We will also write, for example, $abd$ for short of $\{a,b,d\}$.

\section{Heuristics behind the proof}
\subsection{Review of the shattering strategy}
Let us first have a brief review of the proof for the case $4\leq |\mathcal{F}|\leq 113$ and $k=2$. 

We considered a minimal $2$-good set $S$ and showed that by the minimality, for any element $y\in S$, there is $F_y\in \mathcal{F}$ that $F_y\cap S=\{y\}$ because if there was no such $F_y$, $S-y$ became a smaller $2$-good set. 

By the union-closed property, for any $T\subseteq S$, we find that $F_T:= \bigcup_{y\in T} F_y$ is in $\mathcal{F}$, and $F_T\cap S=T$. This shows that $|S|\leq \log_2|\mathcal{F}|$ because these $2^{|S|}$ sets $F_T,T\subseteq S$, are surely different. 

Lastly, using $F_T$'s to analyse the \textit{incidence} $\sum_{A\in \mathcal{F}} |A\cap S|$ of $\mathcal{F}$ over $S$ when $f_2(\mathcal{F})\leq \frac13$, we get that either $s=4$ and $m \geq 45$, or $s=5$ and $m\geq 70.5$.

\subsection{Linear programming and two ways of winning}
We can make a more careful estimate: For each subset $T$ of $S$, we define $q_T$ as the quantity of sets $A$ in $\mathcal{F}$ that $S\cap A=T$.
\begin{itemize}
    \item By $S$ being $2$-good, there are at most $2$ sets that do not intersect $S$, namely $\varnothing$ and $\{1\}$. Thus
    \[q_\varnothing\leq 2.\]
    \item By $F_T\cap S=T$, we obtain 
    \[q_T\geq 1 \qquad\text{for each }T\subseteq S.\] 
    \item As $f_2(\mathcal{F})\leq \frac 13$, we obtain 
    \[\sum_{y\in T}q_T\leq \frac m3 \qquad \text{for each }y\in S.\]
\end{itemize}

It turns out that these are all linear constraints for $q_T$'s and $m$, so we can solve it by setting the following linear programming $L_0$:

\begin{description}
    \item[Variables] $m,q_T:T\subseteq S$
    \item[Minimize] $m$
    \item[Subject to]
    \begin{equation*}
    \begin{array}{r@{\,}ll}
        q_\varnothing &\leq 2 ,&\\
         q_T & \geq 1  ,&T\subseteq S\\
        \sum_{T\subseteq S} q_T &= m,&\\
        \sum_{y\in T}q_T & \leq \frac m3 ,& y\in S.\\
    \end{array}
    \end{equation*}
\end{description}

By solving this directly, we get same bounds $m\geq 45$ for $|S|=4$ and $m\geq 70.5$ for $|S|=5$. So what happens when this LP reaches the lower bound? A solution reaching the lower bound of $m$ when $|S|=4$ is
{\linespread{1}
\[q_{T}=\begin{cases}
    2,&T=\varnothing;\\
    8,&|T|=1;\\
    1,&|T|\geq 2.
\end{cases}\]
}
This configuration can cause many dependencies. For example, while there are many pairs of choices of $F_a, F_b$ in the proof, they all result in the unique $F_{ab}:= F_a\cup F_b$! This phenomenon could only happen when each element in $F_{ab}$ is either a common element of $F_a$s or a common element of $F_b$s. In this scenario, elements \textit{outside} $S$ appear in many sets.

Hence, we can strengthen $L_0$ by considering another carefully chosen element $x$. This $x$ allows two ways of winning: Either
\begin{itemize}
    \item this $x$ helps us find different sets (by spotting $x\notin F_T, x\in F_T'$) that $F_{T}\cap S=T=F_T'\cap S$ for big $T$. In this way, it gives a greater lower bound for $q_T$. Or,
    \item $x$ itself lies in many sets. In this way, it gives an extra constraint for $L_0$, as $x$ cannot lie in over $\frac{m}{3}$ sets if $x\neq 1$.
\end{itemize}

\subsection{The choice of $x$ and three lemmas}
We will pick $x\neq 1$ and an $a\in S$ such that there are two sets $F_a,F_a'$ with $F_a\cap (S+x)=a, F_a'\cap (S+x)=ax$. 

For any choice of $F_y, y\in S$, such that $F_y\cap S=\{y\}$, we associate 
\[ F_{T}:= \bigcup_{y\in T} F_y, T\subseteq S.\]

Now we see how the flexibility of choice of $F_a$ helps us: If we can choose $F_y$ for $y\in S$ such that $x\notin F_T$, then $F_T\cup F_a$ and $F_T\cup F_a'$ are two different sets in $\mathcal{F}$, both of which intersect $S$ at $T+a$, thus $q_{T+a}\geq 2$.

We can do much better than this, but we give some definitions before we can precisely state the lemma:
\begin{definition}[flexible, covered]\mbox{}
    For $y\in S$, we say $a$ is $x$-\textit{flexible} if there are two sets $F_y,F_y'$ with $F_y\cap (S+x)=y, F_y'\cap (S+x)=xy$. 

    We say $y$ is \textit{covered} (by $x$) if all choices of $F_y$ contains $x$. Equivalently, $y$ is covered by $x$ iff $S+x-y$ is $2$-good.
\end{definition}

Let $C$ denote the set of covered elements.
\begin{lemma}\label{lemma:smallway}
    Let $a\in S$ be $x$-flexible. Then $q_T\geq 2$ for each $T\subseteq S$ such that
    \begin{itemize}
        \item either $a\in T$ and $T\cap C=\varnothing$, or
        \item $T$ contains at least $2$ elements in $C$. 
    \end{itemize}
\end{lemma}

On the other hand, if there is a choice of $\{F_y:y\in S\}$ such that $x\in F_T$, then $x\in F_{T'}$ for any $T\subseteq T'\subseteq S$. we can find many sets in $\mathcal{F}$ that contains $x$ if there are many covered elements. It leads to the second lemma:

\begin{lemma}\label{lemma:largeway}
    Let $a\in S$ be $x$-flexible. and $x$ covers the elements in $C$. Then 
    \[(2^{s}-2^{s-1-|C|}) + \sum_{c\in C} q_c - |C| \leq \frac{m}{3}.\]
\end{lemma}

These lemmas handle cases where $|C|$ is small and large well. For $|C|=1$, it turns out that we can benefit from assuming $S$ has the best incidence among all minimal $2$-good $s$-sets of $\mathcal{F}$.

\begin{lemma}\label{lemma:middleway}
    Let $a\in S$ be $x$-flexible. Suppose that $S$ is a minimal $2$-good set which maximizes the incidence $\sum_{A\in \mathcal{F}} |A\cap S|$ among all minimal $2$-good $s$-sets of $\mathcal{F}$. Suppose that $x$ only covers $b\in S$. Then 
    \begin{itemize}
        \item There are $2^{|S|-2}$ sets in $\mathcal{F}$ that intersects $S$ at $\geq 2$ elements, contains $b$ but not $x$.
        \item There are $2^{|S|-2}-1$ sets in $\mathcal{F}$ that intersects $S$ at $\geq 3$ elements, contains $b$ but not $x$.
        \item There are $2^{|S|-3}-1$ sets in $\mathcal{F}$ that intersects $S$ at $\geq 4$ elements, contains $b$ but not $x$.
    \end{itemize}
\end{lemma}

In the next section, we will justify the choice of $a,x$ and prove these three lemmas. Remind that these lemmas are all under the assumption $f_2(\mathcal{F})\leq \frac13$.

To utilise \cref{lemma:middleway}, we will assume that $S$ maximises the incidence among the minimal $2$-good sets with the corresponding size. 

We will separate the analysis regarding the size of $C$. We obtain a bound for $m$ by first obtaining extra constraints for $q_T$s using \cref{lemma:smallway,lemma:largeway,lemma:middleway}, then solve the linear program $L_0$ with these extra constraints. Detailed LP calculations are available in the companion Git repository\footnote{\url{https://github.com/haur576/k-Union-Closed-Sets-Conjecture}}.

A summary of the bounds for $m$ is shown in \cref{fig:summary}. One can find that this is sufficient to reach a contradiction unless $s=4$ and $|C|=0,1$.
\begin{figure}[htbp!]
    \centering
    \[\begin{array}{c|cccc}
        \text{Bound for $m$ if $f_2(\mathcal{F})\leq \frac 13$} & |C|=0 &|C|=1 & |C|=2 & |C|\geq 3\\\hline
        s=4 & \textcolor{red}{m\geq 81} & \textcolor{red}{m\geq 81} & m\geq  114 & \text{infeasible}\\
        s=5 & m\geq 118.5 & m\geq 115.5 & m\geq 122 & m\geq 114\\
        \text{Using Lemma...} &\text{\labelcref{lemma:smallway}}&\text{\labelcref{lemma:smallway,lemma:middleway}}&\text{\labelcref{lemma:smallway,lemma:largeway}}&\text{\labelcref{lemma:largeway}}
    \end{array}\]
    \caption{Summary of bound found in cases}
    \label{fig:summary}
\end{figure}

\section{Proofs of the lemmas}
First of all, we show that there is always a choice of $a$ and $x$. 
\begin{lemma}
    There is $a\in S$ and $x\notin S\cup\{1\}$ such that $a$ is $x$-flexible.
\end{lemma}
\begin{proof}
    By the constraints of $L_0$, we find that $\sum_{a\in S}q_a\geq 40$ for $|S|=5$ and $q_y\geq 8$ for each $y\in S, |S|=4$. Therefore, there is $a\in S$ such that $q_a\geq 8$. So $\{F\in \mathcal{F}| F\cap S=\{a\}\}$ has size $\geq 8$.

    We pick any two sets in this family that differ not only at $1$. Then there is $x \neq 1$ such that $x$ lies in one set but not the other. These two are desired $F_a'$ and $F_a$, respectively.
\end{proof} 

Now, we prove three lemmas introduced in the previous section. We prove \cref{lemma:largeway} first:
\begin{proof}[proof of \cref{lemma:largeway}]
    This lemma is by simple counting: For each $y\in S$, we choose $F_y$ such that $x\in F_y$ if $y=a$ or $y\in C$, and $\notin F_y$ otherwise. By the choice and the definition of covering, $x$ lies in
    \begin{itemize}
        \item $F_T$, where $T$ contains $a$ or an element in $C$, and
        \item any $F\in \mathcal{F}$ such that $F\cap S=\{c\}$ for some $c\in C$.
    \end{itemize}
    
    There are $2^{|S|}-2^{|S|-1-|C|}$ sets from the first case, $\sum_{c\in C}q_c$ sets from the second case and $|C|$ sets from both. As $x$ cannot lie in more than $\frac{m}{3}$ sets, the result follows from the Inclusion-Exclusion Principle. 
\end{proof}

\begin{proof}[proof of \cref{lemma:smallway}]
    The first case is direct: By the definition of non-covered elements, we can pick $F_y$ for $y\notin C$ such that $x\notin F_y$. If $a\in T$ and $T\cap C=\varnothing$, consider $F_T=\bigcup_{y\in T} F_y$ and $F_T'=F_a'\cup F_T$. These two are different as $x\notin F_T, x\in F_T'$.

    For the second case, there is always an $F$ that $F\cap S=T$ with $x\in F$, and we will find another which $x\notin F$. We first show that for any $b,c\in C$, $S+x-b-c$ is \textit{not} $2$-good.
    
    Suppose conversely that $S+x-b-c$ is $2$-good. If $|S|=4$, then it means that there is a $2$-good set with size $3$, which implies $f_2(\mathcal{F})>\frac 13$. Hence we consider $|S|=5$. 
    
    Notice that now any set $F\in \mathcal{F}$ that $F\cap S=\{b,c\}$ also contains $x$. By the same counting method of proving \cref{lemma:largeway}, $x$ lies in $q_b+q_c+q_{bc}+28-3$ sets. Consider the LP $L_0$ along with $q_b+q_c+q_{bc}+28-3\leq \frac{m}{3}$, we get $m\geq 129$, contradicting to the assumption.

    For any $2$ elements $b,c\in C$, since $S+x-b-c$ is not $2$-good, there is a set $G_{bc}\neq \varnothing,\{1\}$ that does not intersect $S+x-b-c$. However, since $S+x-b$ and $S+x-c$ are $2$-good, both $G_{bc}\cap (S+x-b)$ and $G_{bc}\cap (S+x-c)$ are nonempty. This implies $G_{bc}\cap (S+x)=bc$.

    Now for $T$ that contains $2$ or more elements in $C$, we consider $F_T'$ to be the union of
    \begin{itemize}
        \item $G_{bc}$ for each pair $b,c$ of elements in $T\cap C$, and 
        \item $F_y$ for $y\in T\setminus C$ with $x\notin F_y$.
    \end{itemize}
    Then $F'_T\cap S=T$ and $x\notin F_T$, as desired.
\end{proof}
Lastly, we prove \cref{lemma:middleway}.
\begin{proof}[proof of \cref{lemma:middleway}]
    We first show that there are $2^{|S|-2}$ sets in $\mathcal{F}$ containing $b$ but not $x$.
    
    As $x$ only covers $b$, we may choose $F_y$ such that $x\in F_a, x\in F_b$ and $x\notin F_y$ for $y\neq a,b$. Then there are $2^{|S|-2}$ resulting $F_T$ in $\mathcal{F}$ that contains $x$ but not $b$, namely those with $a\in T, b\notin T$. 
    
    Note that $S+x-b$ is a $2$-good set and is minimal. Indeed, for each $y\neq b$ in $S$, $S+x-y$ is not $2$-good, thus $S+x-b-y$ is not $2$-good. (Also, $S+x-b-x=S-b$ is not $2$-good.)
    
    As $S$ maximises the incidence among minimal $2$-good $s$-sets, $S+x-b$ has no larger incidence than $S$. Thus, the frequency of $b$ is at least the one of $x$. That is, there are at least $2^{|S|-2}$ sets which contains $b$ but not $x$.

    In the following we will show that in these $2^{|S|-2}$ sets which contains $b$ but not $x$,
    \begin{itemize}
        \item $2^{|S|-2}$ of them intersect $S$ at $\geq 2$ elements,
        \item $2^{|S|-2}-1$ of them intersect $S$ at $\geq 3$, and
        \item $2^{|S|-3}-1$ of them intersect $S$ at $\geq 4$.
    \end{itemize}
    Let us call the other $|S|-2$ elements $c,d$ and possibly $e$ if $s=5$. for the choice of $\{F_y\}$, we choose $F_y$ such that $x\notin F_y$ for each $y\neq b$.
    
    Among these $2^{|S|-2}$ sets, we pick one $F$ with minimised intersection size of intersection with $S$. 
    \begin{itemize}
        \item If $|F\cap S|=1$, then $F\cap S=\{b\}$ but $F$ does not contain $x$. It contradicts that $b$ is covered.
        \item If $|F\cap S|=2$, suppose that $F\cap S = ab$ (the case for $bc,bd,be$ are similar). We consider the union
        \[G_T:= F\cup \bigcup_{y\in T} F_y \in \mathcal{F}\]
        for any $T$ that does not contain $a$ nor $b$. These $2^{|S|-2}$ sets contain $b$ but not $x$. Also,
        \begin{itemize}
        \item $2^{|S|-2}$ of them intersect $S$ at $\geq 2$ elements,
        \item $2^{|S|-2}-1$ of them intersect $S$ at $\geq 3$ elements, and
        \item $2^{|S|-2}-1-(|S|-2) \geq 2^{|S|-3}-1$ of them intersect $S$ at $\geq 4$ elements.
        \end{itemize}
        Thus, they meet the conditions. \footnote{For the inequality, we used the fact that $2^{|S|-2}-1-(|S|-2) \geq 2^{|S|-3}-1$ when $|S|\geq 4$.}
        
        \item $|F\cap S|=3$: Then all these $2^{|S|-2}$ sets intersect $S$ at $\geq 3$ elements, satisfying the first two requirements. Assume that $F\cap S= abc$ (the cases for $abd, bcd, abe, bce, bde$ are similar). We consider the union
        \[F_T':=F\cup \bigcup_{y\in T} F_y \in \mathcal{F}\]
        for any $T$ does not contain $a,b,c$. These $2^{|S|-3}$ sets contain $b$ but not $x$, and $2^{|S|-3}-1$ of them intersect $S$ at $\geq 4$ elements. Thus, they meet the conditions.
        \item $|F\cap S|\geq 4$: Then all $2^{|S|-2}$ sets meet $S$ at $\geq 4$ elements, satisfying these three conditions.
    \end{itemize}
\end{proof}

\section{$s=4$}\label{sec:s=4}
Now we assume that $S=\{a,b,c,d\}$ and $x,ax\in \mathcal{F}_{S+x}$.
\subsection{When $|C|=0$}
    
    From \cref{lemma:smallway} we have
    \[q_T\geq 2 \text{ for any } a\in T\subseteq S.\] 
    
    Solving the LP $L_0$ along with this additional constraint gives $m\geq 81$.
    
\subsection{When $|C|=1$}
    Assume that $b$ is covered.

    From \cref{lemma:smallway} we have
    \[q_T\geq 2 \text{ for } T =a, ac,ad,acd.\] 
    
    From \cref{lemma:middleway} we have the following inequalities:
    \begin{corollary}\mbox{}
        \begin{equation}
            \begin{aligned}
                \sum_{b\in T\subseteq S, |T|\geq 2} q_T &\geq 7+4,\\
                \sum_{b\in T\subseteq S, |T|\geq 3} q_T &\geq 4+3, \text{ and }\\
                \sum_{b\in T\subseteq S, |T|\geq 4} q_T &\geq 1+1.
            \end{aligned}
        \end{equation}
        \begin{proof}
            We count $F\in \mathcal{F}$ for those that contain both $b$ and $x$ and those that contain only $b$ but no $x$.

            By choosing $\{F_y\}$ such that $x\in F_b$ and $x\notin F_y$ for other $y\neq b$, we find that there are $7,4,1$ $F_T$'s  with $|F_T\cap S|\geq 2,3,4$, respectively, each of which contains both $b$ and $x$.
            
            From \cref{lemma:middleway}, there are $4,3,1$ $F$'s with $|F\cap S|\geq 2,3,4$, respectively, each of which contains $b$ but not $x$.
        \end{proof}
    \end{corollary}
    
    Solving $L_0$ with these $4+3$ extra constraints gives $m\geq 81$.

\subsection{When $|C|=2$}
    Assume that $b,c$ are covered. 
    
    From \cref{lemma:smallway} we get 
    \[q_T \geq 2 \text{ for } T=a, ad, bc, bcd, abc, abcd.\] 
    
    From \cref{lemma:largeway} we get $12 + q_b+q_c \leq \frac{m}{3}$.

    Solving $L_0$ with these two additional constraints gives $m\geq 114$.

\subsection{When $|C|\geq 3$}
    Assume that $b,c,d$ are covered. 
    
    From \cref{lemma:largeway} we get $12 + q_b+q_c+q_d \leq \frac m3$. 
    
    The linear program $L_0$ with this constraint is infeasible.
    
\section{$s=5$}\label{sec:s=5}
Now we assume that $S=\{a,b,c,d,e\}$ and $x,ax\in \mathcal{F}_{S+x}$.
\subsection{When $|C|=0$}

    From \cref{lemma:smallway} we get
    \[q_T\geq 2 \text{ for any } a\in T\subseteq S.\]

    Solving the LP $L_0$ along with this additional constraint gives $m\geq 118.5$.

\subsection{When $|C|=1$}
    We assume $b\in C$.
    
    From \cref{lemma:smallway} we get 
    \[q_T\geq 2 \text{ for any } T\subseteq S \text{ such that }a\in T,b\notin T.\] 

    By \cref{lemma:middleway} we have the following inequalities:
    \begin{corollary}
        \begin{equation}
            \begin{aligned}
                \sum_{b\in T\subseteq S, |T|\geq 2} q_T &\geq 15+8,\\
                \sum_{b\in T\subseteq S, |T|\geq 3} q_T &\geq 11+7, \text{ and }\\
                \sum_{b\in T\subseteq S, |T|\geq 4} q_T &\geq 5+3.
            \end{aligned}
        \end{equation}
        \begin{proof}
            We count $F\in \mathcal{F}$ for those that contain both $b$ and $x$ and those that contain only $b$ but no $x$.

            By choosing $\{F_y\}$ such that $x\in F_b$ and $x\notin F_y$ for other $y\neq b$, we find that there are $15,11,5$ $F_T$'s  with $|F_T\cap S|\geq 2,3,4$, respectively, each of which contains both $b$ and $x$.

            From \cref{lemma:middleway}, there are $8,7,3$ $F$'s with $|F\cap S|\geq 2,3,4$, respectively, that contains $b$ but not $x$.
        \end{proof}
    \end{corollary}
    
    Consider $L_0$ with these $8+3$ extra constraints, and we get $m \geq 115.5$.
    
\subsection{When $|C|=2$}
    We assume $b,c\in C$. 
    
    From \cref{lemma:smallway} we get 
    \[q_T\geq 2 \text{ for } T=a, ad, ae, ade, bc, bcd, bce, bcde, abc, abcd, abce, abcde.\]

    From \cref{lemma:largeway} we get $26 + q_b+q_c \leq \frac{m}{3}$.
    
    Solving $L_0$ with these two types of additional constraints gives $m\geq 122$.
        
\subsection{When $|C|\geq 3$}
    Assume $b,c,d\in C$ are covered.

    From \cref{lemma:largeway} we get $30 + q_b+q_c+q_d-3 \leq \frac{m}{3}$ or $30 + q_b+q_c+q_d+q_e-4 \leq \frac{m}{3}$. Either we get $30 + q_b+q_c+q_d-3 \leq \frac{m}{3}$.

    Solving $L_0$ with these additional constraints gives us $m\geq 114$.

In conclusion, the case $s=5$ is completely solved.

\section{Concluding remark}

In this paper, we further investigated the potential counterexample for \cref{conj:Nagel} and narrowed down the searching range of their sizes from $[45,113]$ to $[81,113]$ and showed that there is some uniformity to the collection of minimal $2$-good sets. 

To further improve the bound, a possible approach is to pick $x$ more carefully. In our proof, we can see that there are many choices for extra elements. With such many choices, one may possibly push the bound further.

Another perspective is to investigate the uniformity of minimal $2$-good sets of the family. To do this, one could answer this question:

\begin{question}
    What can we say about the minimal $2$-good sets if they all have the same size?
\end{question}

To first ignore the asymmetry induced by the most common element $1$, we consider a definition slightly different from $k$-good sets.
\begin{definition}
    Let $\mathcal{F}$ be a family, a \textbf{cover} $S$ of $\mathcal{F}$ is a set that intersects every set in $\mathcal{F}$.

    A \textbf{minimal cover} is a cover whose proper subsets are not covers.

    Let $\mathcal{MC}(\mathcal{F})$ denote the family of minimal covers of $\mathcal{F}$.
\end{definition}
Equivalently, a cover is a vertex cover when we see $\mathcal{F}$ as the edge set of a hypergraph on $n$ vertices. 

In this case, we can characterize when $\mathcal{MC}(F)$ is uniform:
\begin{theorem}
    Let $\mathcal{F}$ be a family, not necessarily be union-closed.
    \begin{enumerate}
        \item $\mathcal{MC}(\mathcal{F})$ is an antichain. In other words, any two elements in $\mathcal{MC}(\mathcal{F})$ are incomparable.
        \item If $\mathcal{F}$ is an antichain, then $\mathcal{MC}(\mathcal{MC}(\F))=\F$.
        \item Subsequently, let $\mathcal{G}$ be the subfamily of minimal elements in $\mathcal{F}$, then $\mathcal{MC}(\F)$ is $k$-uniform if and only if $\mathcal{G}$ is the minimal covers of a $k$-uniform hypergraph.
    \end{enumerate}
\end{theorem}
\begin{proof}
    We only have to prove the second claim. 
    \begin{description}
            \item[$\F\subseteq \mathcal{MC}(\mathcal{MC}(\F))$:] For $F\in \F$, it is a cover of $\mathcal{MC}(\F)$ since every $C\in \mathcal{MC}(\F)$ intersects $F$. To show that $F\setminus\{f\}$ is not a cover of $\mathcal{MC}(\F)$ for each element $f\in F$, we need to find a minimal cover $C\in \mathcal{MC}(\F)$ contained in  $F^{\complement}\cup\{f\}$. Since $\F$ is an antichain, $F^{\complement}\cup\{f\}$ is indeed a cover and therefore such $C$ exists.
            \item[$\mathcal{MC}(\mathcal{MC}(\F))\subseteq \F$:] For $C$ a minimal cover of $\mathcal{MC}(F)$, we show that it covers a set in $\F$. If this is true, $C$ must in $\mathcal{F}$ as sets in $\mathcal{F}$ are covers of $\mathcal{MC}(\F)$.
            
            Were it not, then for each $F\in \F$, there is an element $a_F\in F$ that $C$ does not have. Notice that $\{a_F:F\}$ is a cover of $\F$, it contains a minimal cover of $\F$. It lies in $\mathcal{MC}(\F)$ but does not intersect $C$, giving a contradiction.            
        \end{description}
\end{proof}

We hope to understand the potential counterexamples for \cref{conj:Nagel} when $k=2$ better by understanding the families whose minimal $2$-good sets are uniform.


\bibliographystyle{amsplain}
\providecommand{\bysame}{\leavevmode\hbox to3em{\hrulefill}\thinspace}
\providecommand{\MR}{\relax\ifhmode\unskip\space\fi MR }
\providecommand{\MRhref}[2]{%
	\href{http://www.ams.org/mathscinet-getitem?mr=#1}{#2}
}
\providecommand{\href}[2]{#2}

\end{document}